\documentclass[11pt]{article}
\usepackage[T1]{fontenc}
\usepackage[utf8]{inputenc}
\usepackage[margin=1in]{geometry}
\usepackage{amsmath,amssymb,amsthm}
\usepackage[colorlinks=true,linkcolor=blue,citecolor=blue,urlcolor=blue,breaklinks=true]{hyperref}
\newcommand{\GL}{\mathrm{GL}}

\newtheorem{theorem}{Theorem}[section]
\newtheorem{lemma}[theorem]{Lemma}
\newtheorem{proposition}[theorem]{Proposition}
\newtheorem{corollary}[theorem]{Corollary}
\theoremstyle{definition}
\newtheorem{conjecture}[theorem]{Conjecture}
\theoremstyle{remark}
\newtheorem{remark}[theorem]{Remark}

\newcommand{\Q}{\mathbb{Q}}
\newcommand{\Z}{\mathbb{Z}}
\newcommand{\F}{\mathbb{F}}
\newcommand{\PP}{\mathbb{P}}
\newcommand{\Hyp}{\mathbb{H}}
\newcommand{\ind}{\operatorname{ind}}
\newcommand{\pencil}{\mathcal{P}}

\title{A pencil of quadratic forms in nine variables\\ with no member of Witt index four}

\author{Tony Quertier\\[4pt]
\small Orange Research, Rennes, France\\[2pt]
\small\texttt{tony.quertier@orange.com}}
\date{}

\begin{document}

\maketitle

\begin{abstract}
We exhibit an explicit pair $(A,B)$ of integral symmetric $9\times9$
matrices, defining a nonsingular pair of quadratic forms over $\Q$, such that no member of the rational pencil
$\lambda q_A+\mu q_B$ has Witt index $4$ over $\Q$. This refutes a conjecture from \cite{Que16a}, which predicted that every nonsingular pair in $n$ variables generates a pencil containing a
form of Witt index $\lceil (n-1)/2\rceil$. The obstruction is purely $2$-adic and affects the whole pencil at once: every member has Witt index exactly $3$ over $\Q_2$. The proof is finite and elementary: a parity argument on $\det(\lambda A+\mu B)$, a congruence lemma
reducing $\PP^1(\Q_2)$ to the twelve classes of $\PP^1(\Z/8)$, and a verification at each class, in which the anisotropy
verdict is certified in two independent ways. The twelve ternary residues are displayed modulo $64$, which suffices to reproduce the anisotropy proofs independently. All scripts are provided in the GitHub repository.
\end{abstract}

\section{Introduction}

Let $q_0,q_1$ be quadratic forms in $n$ variables over $\Q$, with
symmetric matrices $Q_0,Q_1$. The Hasse principle for the system
$q_0=q_1=0$ is a classical subject: it holds for arbitrary pairs when
$n\ge13$ \cite{mordell1959integer} and $n\ge11$ \cite{swinnerton1964rational}, and for smooth
intersections when $n\ge9$ \cite{colliot1987intersections} and $n=8$ \cite{heath2018zeros}. The
\emph{effective} problem of exhibiting a rational solution is
much less developed; the algorithm of \cite{Quertier_2016} treats $n\ge13$,
and the author's thesis \cite{Que16a} develops probabilistic methods
for $n\ge11$ and $n\ge9$.

A key step in those methods is the search, inside the pencil
$\pencil_\Q(Q_0,Q_1)=\{\lambda Q_0+\mu Q_1,  (\lambda:\mu)\in
\PP^1(\Q)\}$, for a member of large Witt index over $\Q$: the larger
that index, the larger the totally isotropic subspace
on which the second form is then restricted. Over $\Q_p$ the
anisotropic dimension of a nondegenerate form is at most $4$
\cite{Serre:1993}, so for $n=9$ every member of the pencil has
local Witt index $3$ or $4$ at every finite place; the question is
whether the maximal value $4$ is attained by a single rational member.
The thesis conjectured that it always is.

\begin{conjecture}[{\cite{Que16a}}]
\label{conj:1028}
Let $q_0,q_1$ be quadratic forms in $n$ variables over $\Q$
satisfying Condition~1. Then the pencil $\pencil_\Q(Q_0,Q_1)$
contains a form of Witt index $\lceil (n-1)/2\rceil$ over $\Q$.
\end{conjecture}

Here Condition~1 is the nonsingularity condition of
\cite[Ch.~2]{Que16a} on the \emph{pair}; by \cite[Prop.~2.3.7]{Que16a}
it is equivalent to: $\det(Q_0)\ne0$ and
$\Delta(\lambda)=\det(\lambda Q_0+Q_1)$ has only simple roots over $\overline{\Q}$. The purpose of this paper is to show
that Conjecture~\ref{conj:1028} fails, by an explicit and fully
verifiable example.

\begin{theorem}\label{thm:main}
Let $A,B$ be the integral symmetric matrices of
\textup{(\ref{eq:matrices})} below. Then:
\begin{enumerate}
\item the pair $(q_A,q_B)$ satisfies Condition~1;
\item every member of the local pencil over $\Q_2$, i.e.\ every form
$c(\lambda q_A+\mu q_B)$ with $c\in\Q_2^\times$ and
$(\lambda:\mu)\in\PP^1(\Q_2)$, has Witt index exactly $3$ over $\Q_2$;
\item consequently every form
$\lambda q_A+\mu q_B$ with $(\lambda:\mu)\in\PP^1(\Q)$, has Witt
index at most $3<4=\lceil(9-1)/2\rceil$ over $\Q$, and
Conjecture~\ref{conj:1028} fails for $n=9$.
\end{enumerate}
\end{theorem}

The pair was found by a random search, and the obstruction is concentrated
at the single prime $2$. That 2-adically pairs exist is not new in itself: in Hall's classification of pairs over $p$-adic fields \cite{hall2024pairs} this is the value $H=3$ for $n=9$, realized by explicit local constructions.
What fails is not one exceptional member but the pencil as a whole: all twelve $2$-adic congruence classes of parameters are simultaneously bounded at index $3$.
The proof is finite: a parity argument shows
that $\det(\lambda A+\mu B)$ is odd for every primitive
$(\lambda,\mu)$ (\S\ref{sec:reduction}); a congruence lemma shows that the $\Z_2$-isometry class of
$\lambda A+\mu B$ only depends on the class of $(\lambda:\mu)$ in
$\PP^1(\Z/8)$, which has twelve points (\S\ref{sec:reduction}); and
at each of the twelve classes the ternary residue after splitting off
three hyperbolic planes is anisotropic over $\Q_2$, certified by an
exhaustive enumeration modulo $64$. This is a proof in itself, requiring no Hasse invariant and is confirmed by an independent cross-check
(\S\ref{sec:verification}). Section~\ref{sec:proof} assembles the
argument, and collects some remarks. 

All computations are reproducible from the accompanying PARI/GP
scripts, publicly available at
\url{https://github.com/Fitz-AI/Programme_GP}; every numerical claim below rests on exact
rational or modular arithmetic; no floating point is used anywhere.

\section{Notation and two elementary lemmas}\label{sec:prelim}

Let $k$ be a field of characteristic $\ne2$ and $q$ a nondegenerate
quadratic form on a $k$-vector space $V$. All maximal totally
isotropic subspaces of $V$ have the same dimension, the \emph{Witt
index} $\ind_k(q)$ is the number $w$ of
hyperbolic planes in a Witt decomposition
$q\cong\Hyp^w\perp q_{\mathrm{an}}$ with $q_{\mathrm{an}}$
anisotropic, and $q_{\mathrm{an}}$ is unique up to isometry \cite{Que16a,Lam2005}.

\begin{lemma}[Localization]\label{lem:localization}
For every nondegenerate form $q$ over $\Q$ and every prime $p$,
$\ind_\Q(q)\le\ind_{\Q_p}(q)$.
\end{lemma}

\begin{proof}
If $W\subseteq\Q^n$ is totally isotropic of dimension
$r=\ind_\Q(q)$, then $W\otimes_\Q\Q_p$ is totally isotropic of
dimension $r$ for $q\otimes\Q_p$.
\end{proof}

Note that only this trivial direction is
used; the full Hasse--Minkowski theorem is not needed anywhere in this paper.

\begin{lemma}[Homothety invariance]\label{lem:homothety}
For $c\in k^\times$, $\ind_k(cq)=\ind_k(q)$.
\end{lemma}

\begin{proof}
Since $c\ne0$, a subspace is totally isotropic for $cq$ if and only
if it is totally isotropic for $q$.
\end{proof}

By Lemma~\ref{lem:homothety}, every member of the pencil is, up to a
scalar not affecting Witt indices, of the form $\lambda A+\mu B$ with
$(\lambda,\mu)\in\Z^2$ primitive; we work with such representatives
throughout.

\section{The pencil}\label{sec:pencil}

\begin{equation}\label{eq:matrices}
{\scriptsize\setlength{\arraycolsep}{2.4pt}
\begin{gathered}A=\begin{pmatrix}
-6&0&2&10&1&-3&-6&8&8\\
0&-3&-2&-9&8&8&-5&1&-2\\
2&-2&-2&8&-2&0&1&2&-8\\
10&-9&8&-6&-6&2&-2&-4&3\\
1&8&-2&-6&10&5&4&-1&-7\\
-3&8&0&2&5&-4&2&-4&-9\\
-6&-5&1&-2&4&2&4&-3&7\\
8&1&2&-4&-1&-4&-3&10&3\\
8&-2&-8&3&-7&-9&7&3&-8
\end{pmatrix},\\[4pt]
B=\begin{pmatrix}
-1&-2&-6&-3&-5&-6&-2&-2&3\\
-2&-1&1&-9&0&-1&-10&1&-6\\
-6&1&1&-5&-5&-6&10&1&-6\\
-3&-9&-5&-4&10&0&-1&-8&4\\
-5&0&-5&10&7&-8&-9&-8&3\\
-6&-1&-6&0&-8&-5&-3&3&-3\\
-2&-10&10&-1&-9&-3&3&3&3\\
-2&1&1&-8&-8&3&3&7&9\\
3&-6&-6&4&3&-3&3&9&5
\end{pmatrix}.\end{gathered}}
\end{equation}

\begin{proposition}\label{prop:condition1}
The pair $(q_A,q_B)$ satisfies Condition~1.
\end{proposition}

\begin{proof}
By \cite[Prop.~2.3.7]{Que16a} it suffices that one of the two
determinants be nonzero and that $\Delta$ have simple roots over
$\overline{\Q}$. A direct computation gives
\[
\det A=813229703, \qquad \det B=9479005821,
\]
both nonzero, so the degree of the dehomogenized
$\Delta(\lambda)=\det(\lambda A+B)$ is exactly $9$ and we have $\gcd(\Delta,\Delta')=1$ in $\Q[\lambda]$,.
The exact computation is included in the accompanying scripts.
\end{proof}

\section{Reduction to twelve congruence classes}\label{sec:reduction}

Every point of $\PP^1(\Q_2)$ admits a representative
$(\lambda,\mu)\in\Z_2^2$ with at least one coordinate a unit; we
call such a representative \emph{primitive}. It is unique up to
multiplication by $\Z_2^\times$.

\begin{lemma}\label{lem:odd}
$\Delta(\lambda,\mu)=\det(\lambda A+\mu B)$ is a $2$-adic unit for
every primitive $(\lambda,\mu)\in\Z_2^2$; in particular
$\lambda A+\mu B$ is unimodular over $\Z_2$ for every such pair.
\end{lemma}

\begin{proof}
Matrices $A$ and $B$ have integer coefficients, so $\Delta(\lambda,\mu)\bmod2$ is
defined for $(\lambda,\mu)\in\Z_2^2$ and only depends on
$(\lambda,\mu)\bmod2$; since $\PP^1(\F_2)$ has three points, it
suffices to check the three representatives $(1,0)$, $(0,1)$,
$(1,1)$. Recalling $\det A=813229703$ and
$\det B=9479005821$ from Proposition~\ref{prop:condition1},
\[
\det(A+B)=\Delta(1,1)=25187288207,
\]
all three of $\Delta(1,0)$, $\Delta(0,1)$, $\Delta(1,1)$ are odd.
\end{proof}

\begin{lemma}[Congruence modulo $8$]\label{lem:mod8}
Let $M,M'$ be symmetric matrices over $\Z_2$ with odd determinant
and $M\equiv M'\pmod8$. Then $M$ and $M'$ are
$\Z_2$-isometric: there is $T\in GL_n(\Z_2)$ with
$T^{\mathsf t}MT=M'$.
\end{lemma}

\begin{proof}
Let $M,M'\in M_n(\Z_2)$ be symmetric matrices with odd determinant. Put $M_0=M$ and $n_0=3$. We construct matrices $M_k$, $\Z_2$-isometric to $M$, and integers $n_k\geq 3$ such that
$$
M'-M_k=2^{n_k}E_k,
\qquad E_k\in M_n(\Z_2)\ \text{symmetric}.
$$

Given $M_k$, set $X_k=M_k^{-1}E_k\in M_n(\mathbb{Z}_2)$ (well defined
since $\det M_k$ is a unit) and $T_k=I+2^{n_k-1}X_k$. As $n_k\ge3$,
$T_k\equiv I\pmod 4$, so $T_k\in\mathrm{GL}_n(\mathbb{Z}_2)$. By
symmetry of $M_k$ and $E_k$ we have $E_k=M_kX_k=X_k^{\mathsf t}M_k$,
whence
$$
T_k^{\mathsf t}M_kT_k
=M_k+2^{n_k}E_k+
  2^{2n_k-2}X_k^{\mathsf t}M_kX_k\\
=M'+2^{2n_k-2}X_k^{\mathsf t}M_kX_k.
$$
Thus, on putting
$$
M_{k+1}=T_k^{\mathsf t}M_kT_k,
\qquad
n_{k+1}=2n_k-2,
$$
we have
$$
M'-M_{k+1}\in 2^{n_{k+1}}M_n(\Z_2).
$$

Since $n_0=3$ and $n_{k+1}=2n_k-2$, the sequence $n_k$ tends to
infinity. Moreover,
$
T_k\equiv I\pmod{2^{n_k-1}},
$
so the products $T_0T_1\cdots T_k$ converge in $\GL_n(\Z_2)$ to a
matrix $T$. Passing to the limit gives
$
T^{\mathsf t}MT=M'.
$
\end{proof}

The modulus $8$ is sharp: the rank-one forms $\langle1\rangle$ and
$\langle5\rangle$ are congruent modulo $4$ but are not isometric over
$\Z_2$. The statement is classical 
\cite[\S93]{o1973introduction} or \cite[Ch.~15]{conway1999sphere}.

\begin{corollary}\label{cor:mod64-anisotropic}
Assume, as is the case here by Lemma~\ref{lem:odd}, that
$\lambda A+\mu B$ is unimodular over $\Z_2$ for every primitive
$(\lambda,\mu)\in\Z_2^2$. Then the function $(\lambda:\mu)\mapsto \ind_{\Q_2}(\lambda A+\mu B)$ on $\PP^1(\Q_2)$ depends only on the reduction of $(\lambda,\mu)$ in $\PP^1(\Z/8)$.
\end{corollary}

\begin{proof}
If $(\lambda,\mu)\equiv(\lambda',\mu')\pmod8$, then since $A,B$ have
integer entries, $\lambda A+\mu B\equiv\lambda'A+\mu'B\pmod8$, and both are unimodular. By Lemma~\ref{lem:mod8} they are
$\Z_2$-isometric hence of equal Witt index.
Now let $u\in(\Z/8)^\times$ and lift it to $\tilde u\in\Z_2^\times$.
Then $(\tilde u\lambda)A+(\tilde u\mu)B=\tilde u(\lambda A+\mu B)$
exactly, and any primitive pair $(\lambda',\mu')\equiv u(\lambda,\mu)
\pmod8$ gives, by the first part, a form $\Z_2$-isometric to
$\tilde u(\lambda A+\mu B)$, whose Witt index equals that of
$\lambda A+\mu B$ by Lemma~\ref{lem:homothety}.

\end{proof}

The set $\PP^1(\Z/8)$ has exactly $12$ points. Indeed, if $\mu$ is a unit, the class has a unique representative $(a:1)$ with $a\in\Z/8$, giving eight classes. Otherwise $\lambda$ is a unit, and there are four further classes.We use the following representatives
\begin{gather*}
(0{:}1),\,(1{:}1),\,(2{:}1),\,(3{:}1),\,(4{:}1),\,(5{:}1),\,(6{:}1),\,(7{:}1),\\
(1{:}0),\,(1{:}4),\,(1{:}2),\,(3{:}2).
\end{gather*}

\section{Verification at the twelve classes}
\label{sec:verification}

For each of the twelve representatives $(\lambda:\mu)$ of $\PP^1(\Z/8)$, we compute an exact rational decomposition of $Q=\lambda A+\mu B$ into three hyperbolic planes and a ternary residue.
We prove that the residue is anisotropic by an exhaustive enumeration modulo $64$ and we check it using the
discriminant and the $2$-adic Hasse invariant.

\subsection*{Stage 1: rational decomposition}
An exact rational matrix $P$ with $\det(P)\ne0$ is computed such
that
\begin{equation}\label{eq:certificate}
P\,Q\,P^{\mathsf t}=\Hyp^3\perp R ,
\end{equation}
where the right-hand side is block-diagonal: three
hyperbolic $2\times2$ blocks with zero diagonal and nonzero
off-diagonal entry and $R$ a ternary symmetric
matrix. The identity \textup{(\ref{eq:certificate})} is checked in exact rational arithmetic; no floating-point computation is used 
(the occurring values of $\det P$ are $\pm1,\pm\tfrac13,\pm\tfrac15$; $P$ need not be integral, only invertible over $\Q$). 
The computed residues $R$ have entries in $\Z$ and odd determinant, checked explicitly for all twelve points.

\subsection*{Stage 2: exhaustive enumeration modulo $64$}

The value $64$ is not claimed to be optimal. It is simply a finite precision at which all twelve residues admit no primitive zero modulo $64$.

\begin{lemma}\label{lem:mod64}
Let $R$ be a ternary symmetric matrix with entries in $\Z_2$. If the
form $x\mapsto x^{\mathsf t}Rx$ is isotropic over $\Q_2$,
then there is a vector $v\in\Z_2^3$, not all of whose entries are
even, with $v^{\mathsf t}Rv\equiv0\pmod{64}$.
\end{lemma}

\begin{proof}
Let $x\in\Q_2^3\setminus\{0\}$ such that $x^{\mathsf t}Rx=0$. Put
$m=\min_i v_2(x_i)$ and $v=2^{-m}x.$
Then $v\in\Z_2^3$, at least one coordinate of $v$ is a unit, and
$v^{\mathsf t}Rv=0.$
Reducing $v$ modulo $64$ gives the required primitive vector
$\bar v$.
\end{proof}

So, the \emph{absence} of such a $v$, established by complete enumeration of the
$64^3-32^3=229376$ primitive residue vectors is a proof of anisotropy over $\Q_2$, not a heuristic. It uses no Hasse invariant, no Hilbert symbol and no sign convention.


Table~\ref{tab:certificates} lists, for each representative, $\det P$
and $R\bmod64$.
\begin{table}[h]
\centering
\footnotesize
\begin{tabular}{c c c c}
$(\lambda:\mu)$ & $\det P$ & $R\bmod64$ & verdict\\
\hline
$(0{:}1)$ & $1/3$ & $\begin{pmatrix}14&-7&24\\-7&10&-10\\24&-10&-7\end{pmatrix}$ & anisotropic\\
$(1{:}1)$ & $-1$ & $\begin{pmatrix}-29&-16&-28\\-16&-26&19\\-28&19&-26\end{pmatrix}$ & anisotropic\\
$(2{:}1)$ & $1/5$ & $\begin{pmatrix}19&-10&14\\-10&-17&-31\\14&-31&22\end{pmatrix}$ & anisotropic\\
$(3{:}1)$ & $-1$ & $\begin{pmatrix}22&-13&-26\\-13&14&-19\\-26&-19&-25\end{pmatrix}$ & anisotropic\\
$(4{:}1)$ & $-1$ & $\begin{pmatrix}-5&6&-11\\6&-13&-31\\-11&-31&-31\end{pmatrix}$ & anisotropic\\
$(5{:}1)$ & $-1$ & $\begin{pmatrix}-14&-14&-25\\-14&-25&26\\-25&26&2\end{pmatrix}$ & anisotropic\\
$(6{:}1)$ & $1$ & $\begin{pmatrix}-22&-13&-31\\-13&-9&-22\\-31&-22&18\end{pmatrix}$ & anisotropic\\
$(7{:}1)$ & $-1/5$ & $\begin{pmatrix}29&-1&-18\\-1&-9&-3\\-18&-3&2\end{pmatrix}$ & anisotropic\\
$(1{:}0)$ & $-1$ & $\begin{pmatrix}25&-15&5\\-15&10&17\\5&17&-14\end{pmatrix}$ & anisotropic\\
$(1{:}2)$ & $1$ & $\begin{pmatrix}11&-30&17\\-30&-13&-17\\17&-17&-27\end{pmatrix}$ & anisotropic\\
$(1{:}4)$ & $1$ & $\begin{pmatrix}-3&21&21\\21&26&22\\21&22&23\end{pmatrix}$ & anisotropic\\
$(3{:}2)$ & $1$ & $\begin{pmatrix}-5&-25&7\\-25&13&-8\\7&-8&-27\end{pmatrix}$ & anisotropic\\
\end{tabular}
\caption{The twelve certificates: representative, $\det P$ from
\textup{(\ref{eq:certificate})}, the residue $R$ reduced modulo $64$,
and the Stage~2 verdict.}
\label{tab:certificates}
\end{table}

\begin{proposition}\label{prop:twelve}
For each of the twelve representatives, the certified residue $R$ is
anisotropic over $\Q_2$.
\end{proposition}

\begin{proof}
For each displayed matrix $R\bmod64$, exhaustive enumeration shows
that there is no primitive vector $\bar v\in(\Z/64\Z)^3$ such that $\bar v^{\mathsf t}R\bar v\equiv0\pmod{64}$.
The contrapositive of Lemma~\ref{lem:mod64} implies that $R$ is anisotropic over
$\Q_2$.

\end{proof}

\begin{remark}
As a second, independent verification of Proposition~\ref{prop:twelve},
the $2$-adic Hasse invariant $\varepsilon_2(R)$ and discriminant were computed and the standard isotropy criterion for ternary forms over
$\Q_2$ \cite[Ch.~IV, Th.~6]{Serre:1993} applied.
\end{remark}

\section{Proof of Theorem \ref{thm:main}}\label{sec:proof}

Assertion (1) is Proposition~\ref{prop:condition1}. For (2), by Lemma~\ref{lem:homothety} (applied over $k=\Q_2$) it
suffices to treat $f=\lambda A+\mu B$ with $(\lambda,\mu)\in\Z_2^2$
primitive. Let $(\lambda_0:\mu_0)$ be the representative of its
class in $\PP^1(\Z/8)$ and $f_0=\lambda_0A+\mu_0B$; by
Corollary~\ref{cor:mod64-anisotropic}, $\ind_{\Q_2}(f)=\ind_{\Q_2}(f_0)$. The hyperbolic
decomposition \textup{(\ref{eq:certificate})} shows
$f_0\cong\Hyp^3\perp R$ over $\Q$, hence over $\Q_2$, and $R$ is
anisotropic over $\Q_2$ by Proposition~\ref{prop:twelve}. Since an
anisotropic form contributes no hyperbolic plane, the uniqueness of
the Witt decomposition (\S\ref{sec:prelim}) gives
$\ind_{\Q_2}(f_0)=3$ exactly.
To conclude, part (2)
already gives $\ind_{\Q_2}(f)=3$, then Lemma~\ref{lem:localization} gives
$\ind_\Q(f)\le\ind_{\Q_2}(f)=3<4$, while
Conjecture~\ref{conj:1028} predicts a member of index
$\lceil(9-1)/2\rceil=4$. \qed



\begin{remark} The real place is not the obstruction.
All twelve representatives have real signature $(5,4)$ (exact
computation), hence real Witt index $4$. The obstruction in
Theorem~\ref{thm:main} is purely $2$-adic.
\end{remark}

\begin{remark} The obstruction concerns the Witt index of the pencil
members, not the solvability of the system. The intersection
$q_A=q_B=0$ does have nontrivial rational points, consistently with
the Hasse principle for smooth intersections in $n\ge 9$ variables
\cite{colliot1987intersections}: restricting $q_B$ to a
three-dimensional totally isotropic subspace of $q_A$  (the Witt
index actually available) instead of a four-dimensional one (the
probabilistic algorithm of \cite{Que16a}), one finds for instance the
primitive integer vector
\[
{\scriptsize
\begin{aligned}
x = (&-2023143055309007350974423271763384994,\;
      -704297098135925649317717433999163886,\\
     &-1853836140207868899197216763733254780,\;
      -478782649071877980721472147202524490,\\
     &-2039661165416372301411314506186206708,\;
      \phantom{-}1256917741536325273190865891078359145,\\
     &-1184941076714184060721032491716759075,\;
      \phantom{-}118250923600070411700001932819569088,\\
     &-1074412061229562970510044783385111636),
\end{aligned}}
\]
which satisfies $x^{\mathsf t}Ax=x^{\mathsf t}Bx=0$ exactly. What
Theorem~\ref{thm:main} rules out is the existence of a form in the pencil with four
hyperbolic planes, not the existence of solutions.
\end{remark}

The counterexample is a \emph{local} obstruction: at $p=2$ no member
of the $\Q_2$-pencil attains index $4$. It therefore says nothing
against the local-global form of the question,
\cite[Conjecture~10.2.7]{Que16a} whether the Witt index
attainable in the rational pencil equals the minimum over all places
of the index attainable in the local pencils --- which remains open
and which we now regard as the correct statement.


\bibliographystyle{plain}
\bibliography{bib-179}

@article{Quertier_2016, title={Effective {H}asse principle for the intersection of two quadrics}, volume={19}, DOI={10.1112/S146115701600022X}, number={A}, journal={LMS Journal of Computation and Mathematics}, author={Quertier, Tony}, year={2016}, pages={73–82}}

@phdthesis{Que16a,
url = "http://www.theses.fr/2016CAEN2027",
title = "Résolution de systèmes de deux équations quadratiques",
author = "Quertier, Tony",
year = "2016",
school={Universit{\'e} de Caen},
pages = "1 vol. (128 f.)",
note = "Thèse de doctorat dirigée par Simon, Denis Mathématiques Caen 2016"
}

@book{Serre:1993,
  title={A course in arithmetic},
  author={Serre, Jean-Pierre},
  year={2012},
  publisher={Springer Science \& Business Media}
}

@book{Lam2005,
  author    = {Tsit-Yuen Lam},
  title     = {Introduction to Quadratic Forms over Fields},
  publisher = {American Mathematical Society},
  year      = {2005},
  series    = {Graduate Studies in Mathematics},
  volume    = {67}
}

@article{swinnerton1964rational,
  title={Rational zeros of two quadratic forms},
  author={Swinnerton-Dyer, H},
  journal={Acta Arithmetica},
  volume={9},
  pages={261--270},
  year={1964},
  publisher={Instytut Matematyczny Polskiej Akademii Nauk}
}

@inproceedings{mordell1959integer,
  title={Integer solutions of simultaneous quadratic equations: To {H}elmut {H}asse on his 60th birthday},
  author={Mordell, LJ},
  booktitle={Abhandlungen aus dem Mathematischen Seminar der Universit{\"a}t Hamburg},
  volume={23},
  number={1},
  pages={126--143},
  year={1959},
  organization={Springer}
}

@book{o1973introduction,
  title={Introduction to quadratic forms},
  author={O'Meara, OT},
  volume={117},
  year={1973},
  publisher={Springer}
}

@book{conway1999sphere,
  title={Sphere packings, lattices and groups},
  author={Conway, John H. and Sloane, Neil J. A.},
  volume={290},
  year={1999},
  publisher={Springer New York}
}

@article{colliot1987intersections,
  title={Intersections of two quadrics and Ch{\^a}telet surfaces. {II}.},
  author={Colliot-Th{\'e}lene, J-L and Sansuc, J-J},
  journal={Journal f{\"u}r die reine und angewandte Mathematik},
  volume={374},
  pages={72--168},
  year={1987}
}

@phdthesis{hall2024pairs,
  title={Pairs of Quadratic Forms over p-Adic Fields},
  author={Hall, John R},
  year={2024},
  school={University of Kentucky},
  note = {DOI: 10.13023/etd.2024.128}
}

@article{heath2018zeros,
  title={Zeros of pairs of quadratic forms},
  author={Heath-Brown, DR},
  journal={Journal f{\"u}r die reine und angewandte Mathematik (Crelles Journal)},
  volume={2018},
  number={739},
  pages={41--80},
  year={2018},
  publisher={De Gruyter}
}

\end{document}